\documentclass[12pt,twoside]{amsart}
\usepackage{amssymb,amsmath,amsthm, amscd, enumerate, mathrsfs}
\usepackage{graphicx, hhline}

\title[A characterization of projective 
spaces]{A characterization of projective 
spaces from the Mori theoretic 
viewpoint}
\author{Osamu Fujino and Keisuke Miyamoto}
\date{2020/6/17, version 0.10}
\address{Osamu Fujino \\ Department of 
Mathematics, Graduate School of Science, 
Osaka University, Toyonaka, Osaka 560-0043, Japan}
\email{fujino@math.sci.osaka-u.ac.jp}
\address{Keisuke Miyamoto \\ Department 
of Mathematics, Graduate School of Science, 
Osaka University, Toyonaka, Osaka 560-0043, Japan}
\email{u901548b@ecs.osaka-u.ac.jp}

\DeclareMathOperator{\Nqlc}{Nqlc}
\DeclareMathOperator{\Ker}{Ker}
\DeclareMathOperator{\Supp}{Supp}
\newtheorem{thm}{Theorem}[section]
\newtheorem{lem}[thm]{Lemma}

\newtheorem{cor}[thm]{Corollary}

\theoremstyle{definition}
\newtheorem{step}{Step}
\newtheorem{case}{Case}
\newtheorem{ex}[thm]{Example}
\newtheorem{defn}[thm]{Definition}
\newtheorem{rem}[thm]{Remark}
\newtheorem*{ack}{Acknowledgments}  

\makeatletter
    
    \@addtoreset{equation}{section}
\makeatother

\begin{document}

\maketitle 

\begin{abstract}
We give a characterization of projective spaces 
for quasi-log canonical pairs from 
the Mori theoretic viewpoint. 
\end{abstract}

\tableofcontents

\section{Introduction}\label{f-sec1} 

In this paper, we give a characterization of projective spaces 
for quasi-log canonical pairs 
from the Mori theoretic viewpoint. 
We are mainly interested in singular varieties which 
naturally appear in the minimal model theory of 
higher-dimensional complex projective varieties. 
Although we could not find it explicitly in the literature, 
the following theorem is more or less well known to the 
experts. 

\begin{thm}\label{f-thm1.1}
Let $(X, \Delta)$ be a projective kawamata log terminal 
pair such that 
$-(K_X+\Delta)$ is ample. 
Assume that $-(K_X+\Delta)\equiv rH$ for some Cartier divisor 
$H$ on $X$ with $r>n=\dim X$. 
Then $X$ is isomorphic to $\mathbb P^n$ with $\mathcal O_X(H)\simeq 
\mathcal O_{\mathbb P^n}(1)$. 
\end{thm}

In his lectures on Fano manifolds in Osaka, 
Kento Fujita explained the above theorem and 
asked if it could be generalized. The following theorem is 
an answer to Fujita's question. 

\begin{thm}\label{f-thm1.2}
Let $[X, \omega]$ be a projective quasi-log canonical 
pair such that $X$ is connected. 
Assume that $\omega$ is not nef and 
that $\omega\equiv rD$ for some Cartier divisor 
$D$ on $X$ with $r>n=\dim X$. 
Then $X$ is isomorphic to $\mathbb P^n$ 
with $\mathcal O_X(D)\simeq 
\mathcal O_{\mathbb P^n}(-1)$. 
Moreover, there are no qlc centers of $[X, \omega]$. 
\end{thm}

By combining Theorem \ref{f-thm1.2} with \cite[Theorem 1.1]{fujino-slc}, 
we obtain the following corollary. 

\begin{cor}\label{f-cor1.3}
Let $(X, \Delta)$ be a projective semi-log canonical 
pair such that $X$ is connected. 
Assume that $K_X+\Delta$ is not nef and 
that $K_X+\Delta\equiv rD$ for some Cartier divisor 
$D$ on $X$ with $r>n=\dim X$. 
Then $X$ is isomorphic to $\mathbb P^n$ 
with $\mathcal O_X(D)\simeq 
\mathcal O_{\mathbb P^n}(-1)$ and 
$(X, \Delta)$ is kawamata log terminal. 
\end{cor}

Just after we put this paper on arXiv, St\'ephane Druel and 
Yoshinori Gongyo pointed out that Theorem \ref{f-thm1.1} 
was already generalized as follows: 

\begin{thm}[{\cite[Theorem 1.1]{araujo-druel}}]\label{f-thm1.4} 
Let $X$ be a normal projective variety and let $\Delta$ be an 
effective $\mathbb R$-divisor on $X$ such that $K_X+\Delta$ 
is $\mathbb R$-Cartier. Assume that $-(K_X+\Delta)\equiv rH$ 
for some ample 
Cartier divisor $H$ on $X$ with $r>n=\dim X$. 
Then $n<r\leq n+1$, $(X, \mathcal O_X(H))\simeq 
(\mathbb P^n, \mathcal O_{\mathbb P^n}(1))$, and 
$\deg \Delta=n+1-r$. 
In particular, $(X, \Delta)$ has only kawamata log 
terminal singularities. 
\end{thm}

We note that there are no assumptions on singularities of 
$(X, \Delta)$ in Theorem \ref{f-thm1.4}. On the other hand, 
in Theorem \ref{f-thm1.2} and Corollary \ref{f-cor1.3}, 
we relax the assumption that $-(K_X+\Delta)$ is ample in 
Theorem \ref{f-thm1.1}, 
although we still require some assumptions on singularities of pairs. 

\medskip 

We summarize the contents of this paper. 
In Section \ref{f-sec2}, we give a sketch of proof of Theorem 
\ref{f-thm1.2} for log canonical pairs in order to make 
our main result more accessible. 
In Section \ref{f-sec3}, we collect some basic definitions of 
the minimal model theory of higher-dimensional 
algebraic varieties and the theory of quasi-log schemes. 
In Section \ref{f-sec4}, we prepare three important lemmas 
on quasi-log schemes 
for the proof of Theorem \ref{f-thm1.2}. 
Section \ref{f-sec5} is devoted to the proof of 
Theorem \ref{f-thm1.2} and Corollary \ref{f-cor1.3}. 

\begin{ack}
The authors thank Kento Fujita very much for interesting 
lectures on Fano manifolds and many useful comments. 
They also thank Professors St\'ephane Druel and Yoshinori Gongyo 
for informing them of \cite{araujo-druel}. 
The first 
author was partly supported by JSPS KAKENHI 
Grant Numbers JP16H03925, JP16H06337.
\end{ack}

We will work over $\mathbb C$, the complex number field, throughout 
this paper. In this paper, a {\em{scheme}} means a separated 
scheme of finite type over $\mathbb C$. 
We will use the theory of quasi-log schemes discussed 
in \cite[Chapter 6]{fujino-foundations}. 

\section{Sketch of Proof}\label{f-sec2}

In order to make Theorem \ref{f-thm1.2} more accessible, 
we give a sketch of proof of the following very special case 
of Theorem \ref{f-thm1.2} and Corollary \ref{f-cor1.3}. 
We note that $[X, K_X+\Delta]$ naturally becomes a quasi-log canonical 
pair when $(X, \Delta)$ is a log canonical pair. In this 
section, we will freely use some standard results of the minimal 
model theory for log canonical pairs (see \cite{fujino-fund}). 

\begin{thm}[Theorem \ref{f-thm1.2} for 
log canonical pairs]\label{f-thm2.1}
Let $(X, \Delta)$ be a projective 
log canonical pair with $\dim X=n$. 
Assume that $K_X+\Delta$ is not nef 
and that $-(K_X+\Delta)\equiv rH$ for some Cartier 
divisor $H$ on $X$ with $r>n$. 
Then $X\simeq \mathbb P^n$ with 
$\mathcal O_X(H)\simeq 
\mathcal O_{\mathbb P^n}(1)$. 
\end{thm}

\begin{proof}[Sketch of Proof of Theorem \ref{f-thm2.1}]
Since $K_X+\Delta$ is not nef, 
we have a $(K_X+\Delta)$-negative 
extremal contraction $\varphi\colon X\to W$ by the cone 
and contraction theorem for log canonical pairs 
(see \cite[Theorem 1.1]{fujino-fund}). 
\begin{case}[$\dim W\geq 1$]\label{f-case1} 
We can take an effective $\mathbb R$-Cartier divisor 
$B$ on $W$ with the following properties: 
\begin{itemize}
\item[(i)] $(X, \Delta+\varphi^*B)$ is log canonical 
outside finitely many points, and 
\item[(ii)] there exists a log canonical center $C$ of 
$(X, \Delta+\varphi^*B)$ such that 
$\varphi(C)$ is a point with $\dim C\geq 1$. 
\end{itemize}
In this situation, we obtain that 
$$-(K_X+\Delta+\varphi^*B)|_C\equiv rH|_C$$ and $H|_C$ is 
ample since $\varphi(C)$ is a point. 
Therefore, by the vanishing theorem for quasi-log schemes 
(see Lemma \ref{f-lem4.2} below), 
we obtain $$\chi(C, \mathcal O_C(tH))\equiv 0. 
$$ 
This is a contradiction since $H|_C$ is ample. 
This means that $\dim W\geq 1$ does not happen. 
\end{case}
\begin{case}[$\dim W=0$]\label{f-case2} 
Since $\varphi\colon X\to W$ is a 
$(K_X+\Delta)$-negative extremal contraction, we see that  
$H$ is ample. 
We can explicitly determine $$\chi (X, \mathcal O_X(tH))$$ by 
$-(K_X+\Delta)\equiv rH$ with $r>n$ and 
the vanishing theorem for log canonical pairs 
(see \cite[Theorem 8.1]{fujino-fund}). 
Then we get $H^n=1$ and 
$$
\dim _{\mathbb C} H^0(X, \mathcal O_X(H))=n+1. 
$$ 
Therefore, 
$$
\Delta(X, H)=n+H^n-\dim _{\mathbb C} H^0(X, \mathcal O_X(H))=0
$$ 
holds, 
where $\Delta(X, H)$ is Fujita's $\Delta$-genus of $(X, H)$. 
This implies $X\simeq \mathbb P^n$ with 
$\mathcal O_X(H)\simeq \mathcal O_{\mathbb P^n}(1)$ 
(see \cite[Theorem 2.1]{fujita1} 
or \cite[Theorem 1.1]{kobayashi-ochiai}). 
\end{case}
This is a sketch of the proof of Theorem \ref{f-thm2.1}. 
\end{proof}

When $(X, \Delta)$ is toric, 
Theorem \ref{f-thm2.1} was already established by the first 
author in \cite[Theorem 1.2]{fujino-toric}, 
which is an easy direct consequence 
of \cite[Theorem 0.1]{fujino-notes}. 
In \cite{fujino-notes} and \cite{fujino-toric}, 
the sharp estimate of lengths of extremal rational curves 
plays a crucial role. On the other hand, we will use some vanishing theorems 
for quasi-log schemes in this paper. 

\section{Preliminaries}\label{f-sec3} 

In this section, we collect some basic definitions of 
the minimal model program and the theory of quasi-log schemes. 
For the details, see \cite{fujino-fund} and 
\cite{fujino-foundations}. 

\medskip 

Let us recall singularities of pairs. 

\begin{defn}[Singularities of pairs]\label{f-def3.1}
A {\em{normal pair}} $(X, \Delta)$ consists of a normal 
variety $X$ and an $\mathbb R$-divisor 
$\Delta$ on $X$ such that $K_X+\Delta$ is $\mathbb R$-Cartier. 
Let $f\colon Y\to X$ be a projective 
birational morphism from a normal variety $Y$. 
Then we can write 
$$
K_Y=f^*(K_X+\Delta)+\sum _E a(E, X, \Delta)E
$$ 
with 
$$f_*\left(\underset{E}\sum a(E, X, \Delta)E\right)=-\Delta, 
$$ 
where $E$ runs over prime divisors on $Y$. 
We call $a(E, X, \Delta)$ the {\em{discrepancy}} of $E$ with 
respect to $(X, \Delta)$. 
Note that we can define the discrepancy $a(E, X, \Delta)$ for 
any prime divisor $E$ over $X$ by taking a suitable 
resolution of singularities of $X$. 
If $a(E, X, \Delta)\geq -1$ (resp.~$>-1$) for 
every prime divisor $E$ over $X$, 
then $(X, \Delta)$ is called {\em{sub log canonical}} (resp.~{\em{sub 
kawamata log terminal}}). 
We further assume that $\Delta$ is effective. 
Then $(X, \Delta)$ is 
called {\em{log canonical}} and {\em{kawamata log terminal}} 
if it is sub log canonical and sub kawamata log terminal, respectively. 

Let $(X, \Delta)$ be a normal pair. 
If there exist a projective birational morphism 
$f\colon Y\to X$ from a normal variety $Y$ and a prime divisor $E$ on $Y$ 
such that $(X, \Delta)$ is 
sub log canonical in a neighborhood of the 
generic point of $f(E)$ and that 
$a(E, X, \Delta)=-1$, then $f(E)$ is called a {\em{log canonical center}} of 
$(X, \Delta)$. 
\end{defn}

\begin{defn}[Operations for $\mathbb R$-divisors]\label{f-def3.2}
Let $V$ be an equidimensional 
reduced scheme. 
An {\em{$\mathbb R$-divisor}} $D$ on $V$ is 
a finite formal sum 
$$
\sum _{i=1}^l d_i D_i
$$ 
where $D_i$ is an irreducible reduced closed subscheme 
of $V$ of pure codimension one with 
$D_i\ne D_j$ for $i\ne j$ and 
$d_i$ is a real number for every $i$. 
We put 
$$
D^{<1}=\sum _{d_i<1}d_iD_i, \quad 
D^{=1}=\sum _{d_i=1} D_i, \quad \text{and} 
\quad D^{>1}=\sum _{d_i>1}d_iD_i. 
$$ 
For every real number $x$, $\lceil x\rceil$ is the integer 
defined by $x\leq \lceil x\rceil <x+1$. 
Then we put 
$$
\lceil D\rceil =\sum _{i=1}^l \lceil d_i \rceil D_i 
\quad \text{and}\quad \lfloor D\rfloor =-\lceil -D\rceil. 
$$
\end{defn}

\begin{defn}[$\sim _{\mathbb R}$ and $\equiv$]\label{f-def3.3}
Let $B_1$ and $B_2$ be $\mathbb R$-Cartier 
divisors on a scheme $X$. 
Then $B_1\sim _{\mathbb R} B_2$ means that 
$B_1$ is {\em{$\mathbb R$-linearly equivalent}} 
to $B_2$, that is, 
$B_1-B_2$ is a finite $\mathbb R$-linear combination 
of principal Cartier divisors. 
When $X$ is complete, $B_1\equiv B_2$ means 
that $B_1$ is {\em{numerically equivalent}} to $B_2$. 
\end{defn}

In order to define quasi-log schemes, we need the notion of 
globally embedded simple normal crossing pairs. 

\begin{defn}[{Globally embedded simple normal crossing 
pairs, see \cite[Definition 6.2.1]{fujino-foundations}}]\label{f-def3.4} 
Let $Y$ be a simple normal crossing divisor 
on a smooth 
variety $M$ and let $D$ be an $\mathbb R$-divisor 
on $M$ such that 
$\Supp (D+Y)$ is a simple normal crossing divisor on $M$ and that 
$D$ and $Y$ have no common irreducible components. 
We put $B_Y=D|_Y$ and consider the pair $(Y, B_Y)$. 
We call $(Y, B_Y)$ a {\em{globally embedded simple normal 
crossing pair}} and $M$ the {\em{ambient space}} 
of $(Y, B_Y)$. A {\em{stratum}} of $(Y, B_Y)$ is a log canonical  
center of $(M, Y+D)$ that is contained in $Y$. 
\end{defn}

Let us recall the definition of quasi-log schemes. 

\begin{defn}[{Quasi-log 
schemes, see \cite[Definition 6.2.2]{fujino-foundations}}]\label{f-def3.5}
A {\em{quasi-log scheme}} is a scheme $X$ endowed with an 
$\mathbb R$-Cartier divisor 
(or $\mathbb R$-line bundle) 
$\omega$ on $X$, a proper closed subscheme 
$X_{-\infty}\subset X$, and a finite collection $\{C\}$ of reduced 
and irreducible subschemes of $X$ such that there is a 
proper morphism $f\colon (Y, B_Y)\to X$ from a globally 
embedded simple 
normal crossing pair satisfying the following properties: 
\begin{itemize}
\item[(1)] $f^*\omega\sim_{\mathbb R}K_Y+B_Y$. 
\item[(2)] The natural map 
$\mathcal O_X
\to f_*\mathcal O_Y(\lceil -(B_Y^{<1})\rceil)$ 
induces an isomorphism 
$$
\mathcal I_{X_{-\infty}}\overset{\simeq}{\longrightarrow} 
f_*\mathcal O_Y(\lceil 
-(B_Y^{<1})\rceil-\lfloor B_Y^{>1}\rfloor),  
$$ 
where $\mathcal I_{X_{-\infty}}$ is the defining ideal sheaf of 
$X_{-\infty}$. 
\item[(3)] The collection of reduced and irreducible subschemes 
$\{C\}$ coincides with the images 
of $(Y, B_Y)$-strata that are not included in $X_{-\infty}$. 
\end{itemize}
We simply write $[X, \omega]$ to denote 
the above data 
$$
\bigl(X, \omega, f\colon (Y, B_Y)\to X\bigr)
$$ 
if there is no risk of confusion. 
Note that a quasi-log scheme $[X, \omega]$ is 
the union of $\{C\}$ and $X_{-\infty}$. 
The reduced and irreducible subschemes $C$ 
are called the {\em{qlc strata}} of $[X, \omega]$, 
$X_{-\infty}$ is called the {\em{non-qlc locus}} 
of $[X, \omega]$, and $f\colon (Y, B_Y)\to X$ is 
called a {\em{quasi-log resolution}} 
of $[X, \omega]$. 
We sometimes use $\Nqlc(X, 
\omega)$ to denote 
$X_{-\infty}$. 
If a qlc stratum $C$ of $[X, \omega]$ is not an 
irreducible component of $X$, then 
it is called a {\em{qlc center}} of $[X, \omega]$. 
\end{defn}

\begin{defn}[{Quasi-log canonical 
pairs, see \cite[Definition 6.2.9]{fujino-foundations}}]
\label{f-def3.6}
Let 
$$
\left(X, \omega, f\colon (Y, B_Y)\to X\right)
$$ 
be a quasi-log scheme. 
If $X_{-\infty}=\emptyset$, then 
it is called a {\em{quasi-log canonical pair}}. 
\end{defn}

The following example is very important. 
Example \ref{f-ex3.7} shows that we can treat log canonical 
pairs as quasi-log canonical pairs. 

\begin{ex}[{\cite[6.4.1]{fujino-foundations}}]\label{f-ex3.7}
Let $(X, \Delta)$ be a normal pair such that 
$\Delta$ is effective. 
Let $f\colon Y\to X$ be a resolution of singularities such that 
$$
K_Y+B_Y=f^*(K_X+\Delta)
$$ 
and that $\Supp B_Y$ is a simple normal crossing divisor on $Y$. 
We put $\omega=K_X+\Delta$. 
Then 
$$
\left(X, \omega, f\colon(Y, B_Y)\to X\right)
$$
becomes a quasi-log scheme. 
By construction, $(X, \Delta)$ is log canonical 
if and only if $[X, \omega]$ is quasi-log canonical. 
We note that $C$ is a log canonical 
center of $(X, B)$ if and only if $C$ is a qlc center 
of $[X, \omega]$. 
\end{ex}

For the basic properties of quasi-log schemes, 
see \cite[Chapter 6]{fujino-foundations}. 

\section{Lemmas}\label{f-sec4}

In this section, we prepare three lemmas on 
quasi-log schemes for the proof of 
Theorem \ref{f-thm1.2}. The first one is an easy consequence of 
Fujita's theory of $\Delta$-genus (see \cite{fujita1}, 
\cite{fujita2}, \cite{fujita-book} and 
\cite[Chapter 3]{iitaka}) and the theory of 
quasi-log schemes (see \cite[Chapter 6]{fujino-foundations}). 

\begin{lem}\label{f-lem4.1}
Let $[X, \omega]$ be a projective quasi-log canonical 
pair such that 
$X$ is irreducible with $\dim X=n\geq 1$. 
Let $H$ be an ample Cartier divisor on $X$. 
Assume that $-\omega\equiv rH$ for some $r>n$. 
Then $X\simeq \mathbb P^n$, $\mathcal O_X(H)\simeq 
\mathcal O_{\mathbb P^n}(1)$, 
$r\leq n+1$, and there are no qlc centers of $[X, \omega]$. 
\end{lem}

\begin{proof}
We will use Fujita's theory of $\Delta$-genus (see 
\cite{fujita1}, \cite{fujita2}, 
\cite[Chapter I]{fujita-book}, and \cite[Chapter 3]{iitaka}) and the 
theory of quasi-log schemes (see \cite[Chapter 6]{fujino-foundations}). 
\begin{step}\label{f-4.1-step1}
Let us consider 
$$
\chi (X, \mathcal O_X(tH))=\sum _{i=0}^n (-1)^i \dim _{\mathbb C} 
H^i(X, \mathcal O_X(tH)). 
$$
Since $H$ is ample, it is a nontrivial polynomial of degree $n$. 
Since $$tH-\omega\equiv (t+r)H 
$$
with $r>n$ by assumption, we have 
$$
H^i(X, \mathcal O_X(tH))=0
$$ 
for $i>0$ and $t\geq -n$ by \cite[Theorem 6.3.5 (ii)]{fujino-foundations}. 
Since 
$$
H^0(X, \mathcal O_X(tH))=0
$$ 
for $t<0$ and 
$$
\chi (X, \mathcal O_X)=\dim _{\mathbb C}H^0(X, \mathcal O_X)=1, 
$$
we have 
\begin{equation}\label{f-eq4.1}
\chi (X, \mathcal O_X(tH))=\frac{1}{n!} (t+1)\cdots (t+n). 
\end{equation}
Therefore, we obtain that $H^n=1$ and 
$$
\dim _{\mathbb C} H^0(X, \mathcal O_X(H))=
\chi (X, \mathcal O_X(H))=n+1. 
$$
This means 
$$
\Delta (X, H)=n+H^n-\dim _{\mathbb C} H^0(X, 
\mathcal O_X(H))=0, 
$$ 
where $\Delta(X, H)$ is Fujita's $\Delta$-genus of $(X, H)$. 
Thus we obtain that $X\simeq \mathbb P^n$ and $\mathcal 
O_X(H)\simeq \mathcal O_{\mathbb P^n}(1)$ 
(see \cite[Theorem 2.1]{fujita1} or \cite[Theorem 1.1]{kobayashi-ochiai}). 
\end{step}
\begin{step}\label{f-4.1-step2} 
In this step, we will see that $r\leq n+1$ always hold true. 

We assume that $r>n+1$ holds true. Then, by \cite[Theorem 6.3.5 (ii)]
{fujino-foundations}, we have 
$$
H^i(X, \mathcal O_X(-(n+1)H))=0
$$ 
for $i>0$. 
Therefore, we obtain 
$$
\chi (X, \mathcal O_X(-(n+1)H))=\dim _{\mathbb C}H^0(X, 
\mathcal O_X(-(n+1)H))=0. 
$$ 
On the other hand, by \eqref{f-eq4.1}, we have 
$$
\chi (X, \mathcal O_X(-(n+1)H))=(-1)^n\ne 0. 
$$
This is a contradiction. This means that $r\leq n+1$ always 
holds. 
\end{step}
\begin{step}\label{f-4.1-step3}
In this step, we will see that $[X, \omega]$ has no 
qlc centers. 

Assume that there exists a zero-dimensional 
qlc center $P$ of $[X, \omega]$. 
Then the evaluation map 
$$
H^0(X, \mathcal O_X(-H))\to \mathbb C (P)
$$ 
is surjective since $$H^1(X, \mathcal I_P\otimes 
\mathcal O_X(-H))=0$$ by \cite[Theorem 6.3.5 (ii)]{fujino-foundations}, 
where 
$\mathcal I_P$ is the defining 
ideal sheaf of $P$ on $X$. 
We note that 
$H$ is ample and $-\omega\equiv rH$ with $r>\dim X\geq 1$. 
This means that $$H^0(X, \mathcal O_X(-H))\ne 0. $$ 
This is a contradiction since $H$ is ample. 
Therefore, there are no zero-dimensional qlc centers of $[X, \omega]$. 

Assume that there exists a qlc center $C$ of $[X, \omega]$ 
with $\dim C\geq 1$. 
By \cite[Theorem 6.3.5 (i)]{fujino-foundations}, 
$[C, \omega|_C]$ is a quasi-log canonical pair with 
$\dim C<\dim X$. Since $$-\omega\equiv rH$$ with 
$r>n$, 
we have 
$$
-\omega|_C\equiv rH|_C
$$ 
with $r>n\geq \dim C+1$. 
This contradicts the result established in Step \ref{f-4.1-step2}. 
It means that there are no qlc centers of $[X, \omega]$. 
\end{step}
We finish the proof of Lemma \ref{f-lem4.1}. 
\end{proof}

The second one is an easy lemma on the vanishing theorem 
for quasi-log schemes. 

\begin{lem}[Vanishing theorem for 
quasi-log schemes]\label{f-lem4.2}
Let $[X, \omega]$ be a projective quasi-log scheme 
with $\dim X_{-\infty}=0$ or $X_{-\infty}=\emptyset$. 
Let $L$ be a Cartier divisor on $X$ such that 
$L-\omega$ is ample. 
Then 
$$
H^i(X, \mathcal O_X(L))=0
$$ 
for every $i>0$. 
\end{lem}

\begin{proof}
If $X_{-\infty}=\emptyset$, then the statement is a special 
case of \cite[Theorem 6.3.5 (ii)]{fujino-foundations}. 
Therefore, from now on, we may assume that 
$X_{-\infty}\ne\emptyset$. 

Let us consider the following short exact sequence: 
$$
0\to \mathcal I_{X_{-\infty}}\to \mathcal O_X\to \mathcal O_{X_{-\infty}}\to 0. 
$$ 
Then we obtain a long exact sequence: 
\begin{equation}\label{f-eq4.2}
\cdots \to H^i(X, \mathcal I_{X_{-\infty}}\otimes 
\mathcal O_X(L))\to 
H^i(X, \mathcal O_X(L))\to H^i(X, \mathcal O_{X_{-\infty}}(L))\to 
\cdots. 
\end{equation}
By \cite[Theorem 6.3.5 (ii)]{fujino-foundations}, 
we get 
$$
H^i(X, \mathcal I_{X_{-\infty}}\otimes \mathcal O_X(L))=0
$$ 
for every $i>0$. 
Since $\dim X_{-\infty}=0$ by assumption, 
we have 
$$
H^i(X, \mathcal O_{X_{-\infty}}(L))=0
$$ 
for every $i>0$. 
Therefore, by \eqref{f-eq4.2}, 
we see that 
$$
H^i(X, \mathcal O_X(L))=0
$$ 
holds true for every $i>0$. 
\end{proof}

The final one is a somewhat technical lemma 
(see \cite[Lemmas 3.1 and 3.2]{fujino-relative}). 

\begin{lem}\label{f-lem4.3} 
Let $[X, \omega]$ be a quasi-log canonical pair such that 
$X$ is irreducible and 
let $\varphi\colon X\to W$ be a proper surjective morphism 
onto a quasi-projective variety $W$ with $\dim W\geq 1$. 
Let $P\in W$ be a closed point such that $\dim \varphi^{-1}(P)\geq 1$. 
Then we can construct an effective $\mathbb R$-Cartier 
divisor $B$ on $W$ such that 
$[X, \omega+\varphi^*B]$ is a quasi-log scheme with 
the following properties: 
\begin{itemize}
\item[(i)] $[X, \omega+\varphi^*B]$ is quasi-log canonical 
outside finitely many points, and 
\item[(ii)] there exists a qlc center $C$ of 
$[X, \omega+\varphi^*B]$ such that $\varphi(C)=P$ with $\dim C\geq 1$. 
\end{itemize}
\end{lem}

\begin{proof} 
We divide the proof into several cases. 

\setcounter{case}{0}
\begin{case}\label{f-4.3-step1}
In this case, we assume that 
there are no qlc centers of $[X, \omega]$ in $\varphi^{-1}(P)$. 

Let $f\colon (Y, B_Y)\to X$ be a quasi-log resolution of $[X, \omega]$ 
as in Definition \ref{f-def3.5}. 
We take general very ample Cartier divisors 
$B_1, \ldots, B_{n+1}$ on $W$ such that 
$P\in \Supp B_i$ for every $i$. 
By \cite[Proposition 6.3.1]{fujino-foundations}, 
we may further assume that 
$$
\left(Y, \sum _{i=1}^{n+1} (\varphi\circ f)^*B_i +\Supp B_Y\right)
$$ 
is a globally embedded simple normal crossing pair 
(see \cite[Theorem 3.35]{kollar}). 
By \cite[Lemma 6.3.13]{fujino-foundations}, 
we can take $0<c<1$ with the following 
properties: 
\begin{itemize}
\item[(a)] $\left(B_Y+c\sum _{i=1}^{n+1} (\varphi\circ f)^*B_i\right)^{>1}=0$ 
or 
$\dim f\left(\Supp \left( 
B_Y+c\sum _{i=1}^{n+1} (\varphi\circ f)^*B_i\right)^{>1}
\right)=0$, and 
\item[(b)] there exists an irreducible component $G$ of 
$\left(B_Y+c\sum _{i=1}^{n+1} (\varphi\circ f)^*B_i\right)^{=1}$ 
such that 
$\dim f(G)\geq 1$. 
\end{itemize}
We put $B=c \sum _{i=1}^{n+1} B_i$. 
Then, by construction, 
we see that 
$$
f\colon \left(Y, B_Y+(\varphi\circ f)^*B\right) 
\to [X, \omega+\varphi^*B] 
$$ 
gives a desired quasi-log structure on $[X, \omega+\varphi^*B]$. 
\end{case}
\begin{case}\label{f-4.3-step2}
In this case, we assume that 
there exists a qlc center $C$ of $[X, \omega]$ in $\varphi^{-1}(P)$ 
with $\dim C\geq 1$. 

Obviously, it is sufficient to put $B=0$. 
\end{case}
\begin{case}\label{f-4.3-step3} 
In this case, we assume that every qlc center 
of $[X, \omega]$ contained in $\varphi^{-1}(P)$ is zero-dimensional. 

Let $f\colon (Y, B_Y)\to X$ be a quasi-log resolution 
of $[X, \omega]$ as in Definition \ref{f-def3.5}. 
We take general very ample Cartier divisors $B_1, \ldots, 
B_{n+1}$ on $W$ such that 
$P\in \Supp B_i$ for every $i$ as in Case \ref{f-4.3-step1}. 
Let $X'$ be the union of 
all qlc centers contained in $\varphi^{-1}(P)$. 
By \cite[Proposition 6.3.1]{fujino-foundations}, 
we may assume that 
the union of all strata of $(Y, B_Y)$ mapped to 
$X'$ by $f$, which is denoted by $Y'$, is a union of 
some irreducible components of $Y$. 
We put $Y''=Y-Y'$, $K_{Y''}+B_{Y''}=(K_Y+B_Y)|_{Y''}$, 
and $f''=f|_{Y''}$. 
We may further assume that 
$$
\left(Y'', \sum _{i=1}^{n+1} (\varphi\circ f'')^*B_i +\Supp B_{Y''}\right)
$$ 
is a globally embedded simple normal crossing pair by 
\cite[Proposition 6.3.1]{fujino-foundations} and 
\cite[Theorem 3.35]{kollar}. 
We note that by 
the proof of \cite[Theorem 6.3.5 (i)]{fujino-foundations} 
$$
\mathcal I_{X'}=f''_*\mathcal O_{Y''} 
(\lceil -(B^{<1}_{Y''})\rceil -Y'|_{Y''})
$$ 
holds true, where $\mathcal I_{X'}$ is 
the defining ideal sheaf of $X'$ on $X$. 
We also note that $B_{Y''}\geq Y'|_{Y''}$ by construction. 
By \cite[Lemma 6.3.13]{fujino-foundations}, 
we can take $0<c<1$ with the following properties: 
\begin{itemize}
\item[(c)] $\dim f''\left(\Supp \left( 
B_{Y''}+c\sum _{i=1}^{n+1} (\varphi\circ f'')^*B_i\right)^{>1}
\right)=0$, and 
\item[(d)] there exists an irreducible component $G$ of 
$\left(B_{Y''}+c\sum _{i=1}^{n+1} (\varphi\circ f'')^*B_i\right)^{=1}$ such that 
$\dim f''(G)\geq 1$. 
\end{itemize}
We put $B=c\sum _{i=1}^{n+1}B_i$. 
Then, by construction, 
we see that 
$$
f''\colon \left(Y'', B_{Y''}+(\varphi\circ f'')^*B\right) 
\to [X, \omega+\varphi^*B] 
$$ 
gives a desired quasi-log structure on $[X, \omega+\varphi^*B]$. 
\end{case} 
In any case, we got a desired effective $\mathbb R$-Cartier 
divisor $B$ on $W$. 
We note that $[X, \omega+\varphi^*B]$ is quasi-log 
canonical outside $\varphi^{-1}(P)$ by construction. 
\end{proof}

\section{Proof}\label{f-sec5}

In this section, we will prove Theorem \ref{f-thm1.2} and 
Corollary \ref{f-cor1.3} by using the lemmas obtained in Section \ref{f-sec4}. 

\medskip 

Let us prove Theorem \ref{f-thm1.2}, which is 
the main result of this paper. 

\begin{proof}[Proof of Theorem \ref{f-thm1.2}]
In this proof, we put $H=-D$. 
\setcounter{case}{0}
\begin{case}\label{f-5-step1}
In this case, we assume that $X$ is irreducible. 

Since $\omega$ is not nef, we can take an $\omega$-negative 
extremal contraction $\varphi\colon X\to W$ by the cone 
and contraction theorem of quasi-log canonical pairs 
(see \cite[Theorems 6.7.3 and 6.7.4]{fujino-foundations}). 
If $\dim W\geq 1$, then we can take an effective 
$\mathbb R$-Cartier divisor $B$ on $W$ satisfying 
the properties in Lemma \ref{f-lem4.3}. 
Let $C$ be a qlc center of $[X, \omega+\varphi^*B]$ as in 
Lemma \ref{f-lem4.3}. 
We put $$C'=C\cup \Nqlc(X, \omega+\varphi^*B). 
$$ 
By adjunction (see \cite[Theorem 6.3.5 (i)]{fujino-foundations}), 
$[C', (\omega+\varphi^*B)|_{C'}]$ is a quasi-log scheme. 
We note that there exists the following short exact sequence: 
\begin{equation}\label{f-eq5.1}
0\longrightarrow \Ker\alpha\longrightarrow 
\mathcal O_{C'}\overset{\alpha}\longrightarrow 
\mathcal O_C\longrightarrow 0 
\end{equation} 
such that $\Ker\alpha=0$ or 
the support of $\Ker\alpha$ is zero-dimensional. 
We also note that 
$$
-\left(\omega+\varphi^*B\right)|_{C'}\equiv rH|_{C'}
$$ 
since $\dim \varphi(C')=0$. 
By Lemma \ref{f-lem4.2} and 
\eqref{f-eq5.1},  
$$
H^i(C, \mathcal O_C(tH))=H^i(C', \mathcal O_{C'}(tH))=0
$$ 
for $i>0$ and $t\geq -n$ since $r>n$ by assumption. 
Since $H|_{C}$ is ample, 
we have 
$$
H^0(C, \mathcal O_C(tH))=0
$$ 
for $t<0$. 
This means that 
$$
\chi (C, \mathcal O_C(tH))=0
$$ 
for $t=-n, \ldots, -1$. 
Therefore, we get 
$$
\chi (C, \mathcal O_C(tH))\equiv 0 
$$ 
by $n\geq \dim C+1$. 
This is a contradiction since $H|_C$ is ample. 
This implies that $\dim W=0$ and that $H$ is ample. 
Thus we obtain that $X\simeq \mathbb P^n$ with $\mathcal O_X(D)
\simeq \mathcal O_{\mathbb P^n}(-1)$ by 
Lemma \ref{f-lem4.1}. 
Moreover, there are no qlc centers of $[X, \omega]$ 
by Lemma \ref{f-lem4.1}. 
\end{case}
\begin{case}\label{f-5-step2} 
Let us assume now that $X$ is not necessarily irreducible. 
We take an irreducible component $X'$ of $X$ such that 
$\omega'=\omega|_{X'}$ is not nef. 
By adjunction (see \cite[Theorem 6.3.5 (i)]{fujino-foundations}), 
$[X', \omega']$ is an irreducible 
quasi-log canonical pair such that $\omega'\equiv rD|_{X'}$ with 
$r>n \geq \dim X'$. 
By Case \ref{f-5-step1}, 
we see that $H|_{X'}$ is ample. 
We note that 
$[X', \omega']$ has qlc centers if we assume $X\ne X'$ since 
$X$ is connected (see \cite[Theorem 6.3.11 and 
Theorem 6.3.5 (i)]{fujino-foundations}). 
Therefore, by Lemma \ref{f-lem4.1}, 
we obtain that $X'\simeq \mathbb P^n$, $X'=X$, and 
$\mathcal O_X(D)\simeq 
\mathcal O_{\mathbb P^n}(-1)$. 
In particular, $X$ is always irreducible. 
\end{case}
We finish the proof of Theorem \ref{f-thm1.2}. 
\end{proof}

We prove Corollary \ref{f-cor1.3} as an application 
of Theorem \ref{f-thm1.2}. 

\begin{proof}[Proof of Corollary \ref{f-cor1.3}]
By \cite[Theorem 1.2]{fujino-slc}, $[X, K_X+\Delta]$ naturally 
becomes a quasi-log 
canonical pair. 
Therefore, we obtain the desired statement by Theorem \ref{f-thm1.2}. 
We note that $(X, \Delta)$ is kawamata log terminal 
since there are no qlc centers of $[X, K_X+\Delta]$ 
(see \cite[Theorem 1.2 (5)]{fujino-slc}). 
\end{proof}

We close this paper with a remark on Corollary \ref{f-cor1.3}. 

\begin{rem}\label{rem5.1} 
We can prove Corollary \ref{f-cor1.3} without 
using the theory of quasi-log schemes. 
By taking the normalization and a dlt blow-up 
(see \cite[Theorem 10.4]{fujino-fund} and 
\cite[Theorem 4.4.21]{fujino-foundations}), 
we can reduce the problem to the case where 
$(X, \Delta)$ is a $\mathbb Q$-factorial 
dlt pair. By taking a $(K_X+\Delta)$-negative 
extremal contraction $\varphi:X\to W$ and 
decreasing the coefficients of $\Delta$ slightly, 
we can check that $\dim W=0$ by using the 
argument in the proof of \cite[Theorem 3.1]{andreatta-w} 
(see \cite[Remark 3.1.2]{andreatta-w}). 
Note that some results in \cite{andreatta-w} 
are generalized in \cite{fujino-relative}. 
Then we obtain that $-(K_X+\Delta)$ is ample. 
This implies that $X\simeq \mathbb P^n$ holds 
(see Case \ref{f-case2} in Sketch of Proof of 
Theorem \ref{f-thm2.1} or 
Step \ref{f-4.1-step1} in the proof of Theorem \ref{f-lem4.1}). 
\end{rem}


\end{document}